\documentclass[11pt,twoside,english,reqno,a4paper]{amsart}
\usepackage{listings,graphicx,amsmath,varioref,amscd,ulem, amssymb,color,bm,stmaryrd,amsthm,amsfonts,graphics,geometry,latexsym,pgf,pst-all} 
\theoremstyle{plain}
\usepackage{esint}
\usepackage{color}

\usepackage{amsthm}
\theoremstyle{plain}
\newtheorem{theorem}{Theorem}[section]

\newtheorem{lemma}[theorem]{Lemma}

\newtheorem{corollary}[theorem]{Corollary}

\usepackage{geometry}
\usepackage[colorlinks=false]{hyperref}

\theoremstyle{definition}
\newtheorem{defin}[theorem]{Definition}

\newtheorem{remark}[theorem]{Remark}
\newtheorem{example}{Example}

\theoremstyle{remark}

\usepackage[light,condensed,math]{iwona}
\usepackage[T1]{fontenc}

\numberwithin{equation}{section}

\def\dys{\displaystyle}

\newcommand{\car}[1]{\raise1pt\hbox{$\chi$}_{#1}}

\def\og{\leavevmode\raise.3ex\hbox{$\scriptscriptstyle\langle\!\langle$~}}
\def\fg{\leavevmode\raise.3ex\hbox{~$\!\scriptscriptstyle\,\rangle\!\rangle$}}

\author[F. Oliva]{Francescantonio Oliva}
\address[F. Oliva]{Istituto Nazionale di Alta Matematica (Indam), Dipartimento di Scienze di Base e Applicate per l' Ingegneria, ``Sapienza" Universit\`a di Roma, Via Scarpa 16, 00161 Roma, Italy 
\\ francesco.oliva@sbai.uniroma1.it}
\author[F. Petitta]{Francesco Petitta}
\address[F. Petitta]{Dipartimento di Scienze di Base e Applicate per l' Ingegneria, ``Sapienza" Universit\`a di Roma, Via Scarpa 16, 00161 Roma, Italy 
	\\ francesco.petitta@sbai.uniroma1.it}

\begin{document}

\title{A nonlinear parabolic problem with singular terms and nonregular data}

\begin{abstract}
We study existence of nonnegative solutions to a nonlinear parabolic boundary value problem with a general singular lower order term and  a nonnegative measure as nonhomogeneous datum,  of the form 
$$
\begin{cases}
\dys   u_t - \Delta_p u  = h(u)f+\mu & \text{in}\ \Omega \times (0,T),\\
u=0 &\text{on}\ \partial\Omega \times (0,T),\\
u=u_0 &\text{in}\ \Omega \times \{0\},
\end{cases}
$$
where $\Omega$ is an open bounded subset of $\mathbb{R}^N$ ($N\ge2$), $u_0$ is a nonnegative integrable function, $\Delta_p$ is the $p$-laplace operator, $\mu$ is a nonnegative bounded Radon measure on $\Omega \times (0,T)$ and  $f$ is a nonnegative function of $L^1(\Omega \times (0,T))$. The term $h$ is a positive  continuous function possibly blowing up at the origin. Furthermore, we show uniqueness of finite energy solutions in presence of a nonincreasing $h$.
\end{abstract}

\maketitle
\tableofcontents

\section{Introduction}

Let $\Omega$ be an open bounded subset of $\mathbb{R}^N$; we are mainly concerned with nonnegative  solutions of problems modeled by
\begin{equation}\label{pbi}
\begin{cases}
\dys   u_t - \Delta_p u  = fu^{-\gamma}+\mu & \text{in}\ \Omega \times (0,T),\\
u=0 &\text{on}\ \partial\Omega \times (0,T),\\
u=u_0 &\text{in}\ \Omega \times \{0\},
\end{cases}
\end{equation}
where $p>2-\frac{1}{N+1}$, $u_0$ is a nonnegative integrable function, $\mu$ is a nonnegative bounded Radon measure on $\Omega \times (0,T)$,   $f$ is a nonnegative function in  $L^1(\Omega \times (0,T))$, and $\gamma>0$.

The interest in problems as  \eqref{pbi} (with $p=2$ and smooth data) started in \cite{fw} in connection with the study of   thermo-conductivity  (${u^{\gamma}}$ represented  the resistivity of the material), and later  in the study of signal transmissions  and in the theory non-Newtonian pseudoplastic fluids (\cite{no,nc,gow}). 

From the purely mathematical point of view, a complete setting of the  theory, in the stationary case with smooth data, was developed over the years starting  by the seminal papers \cite{stu,crt}, until the remarkable improvements given in  in \cite{lm} ($p=2$ and $\mu=0$).  

Again in the stationary case with $p=2$ the weak theory was settled in \cite{bo} (see also \cite{bgh,gmm,gmm2})) while both existence and uniqueness in presence of nonlinear operators was proven in \cite{cst,op,op2,sz}. The case of possibly measure as data was faced in \cite{op, po, do}

\medskip 

In the  parabolic setting the literature for problems as in \eqref{pbi} is, by far, more limited.  If $f\equiv 0$ (the case quasilinear case with measure data) we refer to \cite{ppp} (see also \cite{dpp, p,pp}) for a complete account on existence and uniqueness  in the context of renormalized solutions. Moreover,  the case of a  bounded zero-order nonlinearity  has been  treated in \cite{lj}. Both weak and strong regularity of the, so called, SOLA solutions (solutions obtained as limit of approximations) have been also  obtained   (see \cite{bdgo,km,bdp} and references therein). 

  Finally, concerning the singular model case,  for suitably  smooth data $f$ and $\mu$, the existence of  solutions to problems as in \eqref{pbi} was investigated in \cite{dbg} (see also \cite{bep,dbdc,dbg0}).  

\medskip

In this note we consider nonnegative integrable data $f$ and $u_0$, a  nonnegative bounded Radon measure as nonhmogeneous source and a merely continuous, and possibly singular at the origin, nonlinear zero-order term  $h(s)$. Under these general assumptions we prove existence of a nonnegative solution for problem 
$$
\begin{cases}
\dys   u_t - \Delta_p u  = h(u)f+\mu & \text{in}\ \Omega \times (0,T),\\
u=0 &\text{on}\ \partial\Omega \times (0,T),\\
u=u_0 &\text{in}\ \Omega \times \{0\},
\end{cases}
$$
and uniqueness of finite energy solutions (in the homogeneous case) provided $h$ is nonincreasing.

\medskip

The plan of the paper is as follows: in Section \ref{1} we present our main assumptions and we state the main existence result. Section \ref{main} is devoted to the proof of the existence result; the approximation scheme is presented and the basic a priori estimates are obtained (Section \ref{approx_sec}), then the passage to the limit is performed in Section   \ref{pass}, while in Section \ref{comment} some further regularity issues are discussed.  Finally, in Section  \ref{unicas},  a uniqueness result is presented. 
 
\medskip\medskip\medskip\medskip

\subsection*{Notations.}  The parabolic cylinder is denoted by $Q= \Omega \times (0,T)$  (by $Q_t= \Omega \times (0,t)$ for a generic $t>0$), while its lateral surface is $\Gamma=\partial\Omega\times (0,T)$. We denote by $\mathcal{M}(Q)$  the space of Radon measures  with bounded total variation on $Q$.   
We will denote with $r^*= \frac{rN}{N-r}$ the Sobolev conjugate of $1 \le r < N$, while $r'= \frac{r}{r-1}$ indicates the H\"older conjugate of $r>1$.
\\For fixed $k>0$ we will made use of the truncation functions $T_{k}$ and $G_{k}$ defined, resp.,   as 
$
T_k(s)=\max (-k,\min (s,k)), 
$
and 
$
G_k(s)=(|s|-k)^+ \operatorname{sign}(s). 
$
For $\eta,\delta>0$, we define 
\begin{equation}\label{Ttilde}
\tilde{T}_{k,\eta}(s):=\int_{0}^{s} T_{k}^{\eta}(t) \ dt,
\end{equation}
and 
\begin{align}\label{Vdelta}
\displaystyle
V_{\delta}(s):=
\begin{cases}
1 \ \ &s\le \delta, \\
\displaystyle\frac{2\delta-s}{\delta} \ \ &\delta <s< 2\delta, \\
0 \ \ &s\ge 2\delta.
\end{cases}
\end{align}

 For the sake of simplicity we will use the simplified notations $$\int_Q f = \int_{0}^{T}\int_\Omega f=\int_{0}^{T}\int_\Omega f(x,t)\ dxdt\,,$$ and
$$\int_\Omega f = \int_\Omega f(x,t)\ dx \,,$$
when referring to integrals when no ambiguity on the variable of integration is possible.

Finally we denote by 
$$\Omega_\epsilon:= \{x\in \Omega: {\rm dist}(x, \partial\Omega)<\epsilon\}\,,$$
which is well defined for a sufficiently smooth $\partial\Omega$, say Lipschitz. 

\medskip

If no otherwise specified, we will denote by $C$ several constants whose value may change from line to line and, sometimes, on the same line. These values will only depend on the data but they will never depend on the indexes of the sequences we will often introduce.

\section{Setting and main existence result}\label{1}

Let $\Omega$ be a bounded  and smooth open subset of $\mathbb{R}^N (N\ge 2)$, and let $T>0$.  We consider the following nonlinear parabolic problem 

\begin{equation}
\begin{cases}
\displaystyle u_t -\operatorname{div}(a(x,t,\nabla u)) = h(u)f + \mu &  \text{in}\, Q, \\
u(x,t)=0 & \text{on}\ \Gamma,\\
u(x,0)=u_0(x) &  \text{in}\, \Omega, 
\label{pb}
\tag{P}
\end{cases}
\end{equation}

\noindent where $f\ge 0$ belongs to $L^1(Q)$, $\mu \ge 0$ belongs to $\mathcal{M}(Q)$, $\displaystyle{a(x,t,\xi):\Omega\times(0,T)\times\mathbb{R}^{N} \to \mathbb{R}^{N}}$ is a Carath\'eodory function satisfying
\begin{align}
&a(x,t,\xi)\cdot\xi\ge \alpha|\xi|^{p}, \ \ \ \alpha>0,
\label{cara1}\\
&|a(x,t,\xi)|\le \beta|\xi|^{p-1}, \ \ \ \beta>0,
\label{cara2}\\
&(a(x,t,\xi) - a(x,t,\eta )) \cdot (\xi -\eta) > 0,
\label{cara3}	
\end{align}
for every $\xi\neq\eta$ in $\mathbb{R}^N$ and for almost every $(x,t)\in Q$,  with $2-\frac{1}{N+1}<p<N$. A prototype of the  operators we consider is the usual $p$-laplacian defined as $-\operatorname{div}(|\nabla u|^{p-2}\nabla u)$. The initial datum $u_0$ is nonnegative and it belongs to $L^1(\Omega)$. The bound from below on $p$, even if  technical, is standard as it ensures the gradient of the solution to  belong  to $L^1 (0,T; L^1_{\rm loc}{(\Omega)})$.  
 The term $h:[0,\infty)\to [0,\infty]$ is a continuous and possibly singular function with $h(0)\not=0$ which it is finite outside the origin and such that
\begin{equation}
\exists \gamma\ge 0, C,s_0>0: h(s)\le \frac{C}{s^\gamma} \ \text{for all } s\le s_0\,,
\label{h1}\tag{h}
\end{equation}
and such that $h$ is bounded  in $[s_0,\infty)$. 

In the following it will be useful the introduction of the notation $\sigma:=\max(1,\gamma)$.
Let us give  the notion of solution we shall  consider from now on
\begin{defin}
A {\it distributional solution} of problem \eqref{pb} is a function $u\in L^1(0,T;W^{1,1}_{\rm loc}(\Omega))$ with  both  $|a(x,t,\nabla u)| $ and $ h(u)f$ belonging to $L^1(0,T;L^1_{\rm loc}(\Omega))$,  such that 

	\begin{equation}\label{weakdef2}
		\lim_{\epsilon \to 0}\frac{1}{\epsilon}\int_{\Omega_\epsilon}T_k(u) = 0 \ \ \text{for a.e. }t\in (0,T), \ \forall k>0\,,
	\end{equation}
	
	and
	
	\begin{equation} \displaystyle -\int_{Q} u\varphi_t  - \int_{\Omega} u_0\varphi(x,0) + \int_{Q}a(x,t,\nabla u) \cdot \nabla \varphi  =\int_{Q} h(u)f\varphi  +\int_{Q} \varphi d\mu,\label{weakdef3}\end{equation}
	for every $\varphi \in C^1_c(\Omega\times [0,T)).$	
	\label{weakdef}
\end{defin}
\begin{remark} 
Let us remark  that (\ref{weakdef2}) is the weak way we recover that $u=0$ on $\Gamma$. Condition \eqref{weakdef2} is known to be weaker of the classical request to have a solution lying in a space with zero Sobolev trace and it allows to unify the discussion of both $\gamma\leq 1$ and $\gamma>1$ (the same was done, in the stationary case, in \cite{op2}) and to avoid truncations in the definition. We stress that this is the only point where we exploit  the regularity of $\partial \Omega$. If we only assume $\Omega$ to be an open and bounded subset of $\mathbb{R}^N$ then everything works fine provided, 
in the case $\gamma>1$, the boundary condition is intended as $T_{k}^{\frac{\gamma-1+p}{p}} (u)\in L^p(0,T;W^{1,p}_0(\Omega))$, for any $k>0$ i.e. a suitable power of every truncation of $u$ lies, for almost every $t\in (0,T)$, in a Sobolev space with zero classical trace.  If $\gamma\leq 1$ then one may assume  $T_{k}(u)\in L^p(0,T;W^{1,p}_0(\Omega))$, for any $k>0$   (this is how, for instance, the boundary condition was given in \cite{bo,op,po}).  
\end{remark}
\begin{theorem}
	Let $a$ satisfy \eqref{cara1}, \eqref{cara2}, \eqref{cara3}, let $h$ satisfy \eqref{h1} and suppose that $f \in L^1(Q)$, $u_0\in L^1(\Omega)$,  and $\mu \in \mathcal{M}(Q)$ are nonnegative. Then there exists a nonnegative distributional solution $u$ of problem \eqref{pb}.
	\label{theo1}
\end{theorem}

\section{Existence of a distributional solution}
\label{main}
In this section we prove Theorem \ref{theo1}.  To deduce the existence of a distributional solution we work by approximation. First we introduce the approximating scheme and we get basic a priori estimates on the approximating solutions. Then we pass to the limit,  the main difficulty relying  in carefully treat the nonlinear term on the set where  the approximating solutions vanish  and in recovering the boundary datum. At the end of this section we provide some further regularity results on the solution we obtained. 

\subsection{Approximation scheme and a priori estimates}\label{approx_sec}
Consider  the following scheme of approximation
\begin{equation}
\begin{cases}
\displaystyle (u_n)_t -\operatorname{div}(a(x,t,\nabla u_n)) = h_n(u_n)f_n + \mu_n &  \text{in}\, Q, \\
u_n(x,t)=0 & \text{on}\ \Gamma,\\
u_n(x,0)=u_{0_n}(x) &  \text{in}\, \Omega, 
\label{pbapprox}
\end{cases}
\end{equation}
where $h_n(s):= T_n(h(s)), f_n:= T_n(f), u_{0_n}(x):=T_n(u_0(x))$ and $\mu_n$ is a sequence of smooth  functions, bounded in $L^1(Q)$, that converges in the narrow topology of measures to $\mu$. The existence of such a sequence $\mu_n$ is obtained by standard convolution arguments.

 First of all we need to show the existence of a weak solution to \eqref{pbapprox}. The proof, which is based on the Schauder fixed point theorem, is quite standard but we sketch it for completeness. 
\begin{lemma}
	Let $a$ satisfy \eqref{cara1}, \eqref{cara2} and \eqref{cara3}. Then, for any fixed $n\in\mathbb{N}$, there exists a nonnegative solution $u_n \in L^p(0,T;W^{1,p}_0(\Omega)) \cap L^\infty(Q)$ such that $(u_{n})_t \in L^{p'}(0,T;W^{-1,p'}(\Omega)) $ to problem \eqref{pbapprox}.
	\label{exappr} 
\end{lemma}
\begin{proof}
 Let $n\in\mathbb{N}$ be fixed,  $v \in L^p(Q)$ and  consider the following problem
\begin{equation}
\begin{cases}
\displaystyle w_t - \operatorname{div}(a(x,t,\nabla w)) =h_n(v)f_n + \mu_n &  \text{in}\, Q, \\
w(x,t)=0 & \text{on}\ \Gamma,\\
w(x,0)=u_{0_n}(x) &  \text{in}\, \Omega.
\end{cases}
\label{pbs}
\end{equation}
It follows from classical theory (see \cite{lions}) that problem \eqref{pbs} admits a unique solution $w \in L^p(0,T;W^{1,p}_0(\Omega)) \cap C([0,T];L^2(\Omega))$ such that $w_t \in L^{p'}(0,T;W^{-1,p'}(\Omega))$ for every fixed $v \in L^p(Q)$. Furthermore $w$ belongs to $L^\infty(Q)$. \\
Our aim is to prove the existence of a fixed point for the  map
$$G:L^p(Q) \to L^p(Q)\,,$$
which for  any $v \in L^p(Q)$ gives  the weak solution $w$ to \eqref{pbs}.

We take $w$ as test function in the weak formulation of \eqref{pbs} obtaining
\begin{equation}
\displaystyle \int_{0}^{T} \langle w_t, w \rangle + \int_{Q}a(x,t,\nabla w)\cdot \nabla w = \int_{Q}h_n(v)f_nw + \int_{Q}\mu_nw.
\label{sc0}
\end{equation} 
By \eqref{cara1} and by classical integration by parts formula, one has 
$$\displaystyle \frac{1}{2}\int_{\Omega}w^2(x,T) - \frac{1}{2}\int_{\Omega}u_{0_n}^2(x) + \alpha\int_{Q}|\nabla w|^p \le \int_{0}^{T} \langle w_t, w \rangle + \int_{Q}a(x,t,\nabla w)\cdot \nabla w.$$
For the right hand side of \eqref{sc0} by  the H\"{o}lder inequality
$$ \displaystyle \int_{Q}h_n(v)f_nw + \int_{Q}\mu_nw \le Cn^2 \left(\int_{Q}|w|^p\right)^\frac{1}{p} + C||\mu_n||_{L^\infty(Q)}\left(\int_{Q}|w|^p\right)^\frac{1}{p},$$
and so, dropping a positive term
\begin{equation}
\displaystyle \alpha\int_{Q}|\nabla w|^p \le C\left(n^2 + ||\mu_n||_{L^\infty(Q)}\right) \left(\int_{Q}|w|^p\right)^\frac{1}{p} + \frac{1}{2}\int_{\Omega}u_{0_n}^2.
\label{sc1}
\end{equation}
Poincar\'e inequality implies for some constant $C$
$$ \displaystyle \int_{Q}|w|^p \le C\left(n^2 + ||\mu_n||_{L^\infty(Q)}\right) \left(\int_{Q}|w|^p\right)^\frac{1}{p} + \frac{C}{2}\int_{\Omega}u_{0_n}^2,$$
which implies
\begin{equation} \label{star}\displaystyle \left(\int_{Q}|w|^p\right)^\frac{1}{p} \le C.\end{equation}
The constant $C$ is independent of $v$, and so the ball  $B:=B_C (0)$ of  $L^p(Q)$ of radius $C$  is invariant for the map $G$.

 Now we  check the continuity of the map $G$. Let $v_k$ be a sequence of functions  converging to $v$ in $L^p(Q)$.
 
 By the dominated convergence theorem one has that $\displaystyle h_n(v_k)f_n +\mu_n$ converges to $\displaystyle h_n(v)f_n+\mu_n$ in $L^p(Q)$. 
Hence, by uniqueness, one deduces that $w_k:=G(v_k)$ converges to $w:=G(v)$ in $L^p(Q)$. 

 Lastly we need $G(B)$ to be relatively compact.  Let $v_k$ be a bounded sequence, and let $w_k=G(v_k)$.
Reasoning as to obtain \eqref{sc1}, we have
$$\displaystyle \alpha\int_{Q}|\nabla w_k|^p \le C\left(n^2 + ||\mu_n||_{L^\infty(Q)}\right) \left(\int_{Q}|w_k|^p\right)^\frac{1}{p} + \frac{1}{2}\int_{\Omega}u_{0_n}^2,$$
where C is clearly independent from $v_k$.
Recalling \eqref{star} this means that $w_k$ is bounded with respect to $k$ in $L^p(0,T;W^{1,p}_0(\Omega))$. We also deduce from the equation  that $(w_{k})_t$ is bounded with respect to $k$ in $L^{p'}(0,T;W^{-1,p'}(\Omega))$. Hence $w_k$ admits a strongly convergent subsequence in $L^p(Q)$ (see \cite{simon}). This concludes the proof.
\end{proof}
In the next lemma we collect the basic  estimates on $u_n$ which will allow us  to pass to the limit in the  approximating formulation  \eqref{pbapprox}. 
 \begin{lemma}\label{stimaloc}
 		Let $ f\in L^1(Q), \mu \in \mathcal{M}(Q), u_0 \in L^1 (\Omega)$ be nonnegative, let $a$ satisfy \eqref{cara1}, \eqref{cara2} and \eqref{cara3} and let $h$ satisfy \eqref{h1}. Then the sequence $u_n$ of solutions of  \eqref{pbapprox} is bounded in $L^\infty(0,T; L^1(\Omega))$ and in $L^q(0,T; W^{1,q}_{\rm loc}(\Omega))$ with $q<p-\frac{N}{N+1}$. Moreover one has
 		\begin{equation}\label{stimatk}
 \int_{Q}|\nabla T_k(u_n)|^p\varphi^p\le Ck, \ \ \forall k>0,
 		\end{equation}
 		for every nonnegative $\varphi\in C^1_c(\Omega\times[0,T))$.
 \end{lemma}
 \begin{proof}
   	For a fixed $t\in (0,T]$ we take $\displaystyle T_1^\sigma(u_n)$ as a test function in the formulation of $u_n$ on $Q_t$ (recall that $\sigma:=\max(1,\gamma)$);   one has
   	\begin{equation*}\label{stimeloc1}
   	\begin{aligned}
   		\displaystyle &\int_{0}^{t} \langle(u_n)_t, T_1^\sigma(u_n)\rangle + \alpha\sigma\int_{Q_t}|\nabla T_1(u_n)|^p T_1(u_n)^{\sigma-1} \le  \\
   		 &\int_{Q_t\cap\{u_n \le s_0\}}f_n T_1^{\sigma-1}(u_n) +  
   		 \int_{Q_t\cap\{u_n > s_0\}}h_n(u_n)f_n T_1^{\sigma}(u_n) + \int_{Q_t}\mu_n T_1^{\sigma}(u_n),   		
   	\end{aligned}	
   	\end{equation*}
   	and then
  	\begin{equation*}\label{stimeloc2}
	\begin{aligned}
   		\displaystyle \int_{Q_t} (\tilde{T}_{1,\sigma}(u_n))_t \le 	||f||_{L^1(Q)}\left(1+\sup_{s\in[s_0,+\infty)} h(s)\right) + ||\mu_n||_{L^1(\Omega)},
   	\end{aligned}	
\end{equation*}
where $\tilde{T}_{1,\sigma}(s)$ is defined in \eqref{Ttilde}. Furthermore, observing that $\tilde{T}_{1,\sigma}(s)\ge s-1$, one gets  
   	\begin{equation}
   		\displaystyle \int_{\Omega} u_n(x,t) \le ||f||_{L^1(Q)}\left(1+\sup_{s\in[s_0,+\infty)} h(s)\right) + ||\mu_n||_{L^1(Q)} + |\Omega| + \int_{\Omega} \tilde{T}_{1,\sigma}(u_0).
   		\label{stimeloc5}
   	\end{equation}
   	Every term on the right hand side of \eqref{stimeloc5} is also bounded with respect to $t$; hence taking the supremum on $t$ we obtain that 
   	 \begin{gather}
   	 \displaystyle ||u_n||_{L^\infty(0,T;L^1(\Omega))}\le C.
   	 \label{stimapriori3}
   	 \end{gather}
   Now let $\varphi\in C^1_c(\Omega\times[0,T))$ be nonnegative and take $\displaystyle (T_k(u_n)-k)\varphi^p$ as a test function in \eqref{pbapprox}
   \begin{equation*}
   \displaystyle \int_{0}^{T} \langle (u_n)_t, (T_k(u_n)-k)\varphi^p\rangle + \int_{Q}|\nabla T_k(u_n)|^p\varphi^p  + p \int_{Q} a(x,t,\nabla u_n)\cdot \nabla \varphi \varphi^{p-1} (T_k(u_n)-k) \le 0,
   \end{equation*}	 
   which gives 
      \begin{equation*}
   \displaystyle \int_{Q}|\nabla T_k(u_n)|^p\varphi^p \le \epsilon\int_{Q}|\nabla T_k(u_n)|^p\varphi^p + C_\epsilon kp \int_{Q} |\nabla\varphi|^p + C{k},
   \end{equation*}	
  where we also used that  
   $$\left|\int_{0}^T \langle(u_n)_t, (T_k(u_n)-k)\varphi^p \rangle\right|\leq \left|\int_{\Omega} (\tilde{T}_{k,1} (u_{0_n}) -k u_{0_{n}})\varphi^p(x,0)\right| + p\left|\int_Q (\tilde{T}_{k,1} (u_{n}) -k u_{n})\varphi^{p-1}\varphi_t\right| \le Ck.$$ 
   Hence one  has that 
   $$\int_{Q}|\nabla T_k(u_n)|^p\varphi^p\le Ck,$$
   where the constant $C$ does not depend on $n$. By taking  $\displaystyle (T_1(G_k(u_n))-1)\varphi^p$ as a test function in \eqref{pbapprox} where $\varphi\in C^1_c(\Omega\times[0,T))$ one may also deduce 
   \begin{equation}\label{stimastrisce}
   	  \int_{Q\cap \{k<u_n<k+1\}}|\nabla u_n|^p\varphi^p\le C.
   \end{equation}
   From \eqref{stimapriori3} and \eqref{stimastrisce} one can reason as in  \cite[Theorem $4$]{bg} in order to deduce that $u_n$ is bounded in $L^q(0,T; W^{1,q}_{\rm loc}(\Omega))$ with $q<p-\frac{N}{N+1}$.
   \end{proof}
The estimates on $u_n$ imply that the (possibly) singular term is locally bounded in $L^1(Q)$.  In fact, we have 
\begin{corollary}\label{stimaL1}
	Under the assumptions of Lemma \ref{stimaloc} one has that 
	\begin{equation}\label{limitatoL1}
		\int_Q h_n(u_n)f_n\varphi \le C,
	\end{equation}
 		for every nonnegative $\varphi\in C^1_c(\Omega\times [0,T))$ with  $C$ not depending on $n$.
\end{corollary}
\begin{proof}
	Let us take $0\le \varphi\in C^1_c(\Omega\times [0,T))$ as a test function in \eqref{pbapprox} obtaining
	\begin{equation*}
	\begin{aligned}
			\int_Q h_n(u_n)f_n\varphi &\le  \beta\int_Q |\nabla u_n|^{p-1}|\nabla \varphi| 
			&\le  C\int_{{supp (\varphi)}} |\nabla u_n|^q +C,
	\end{aligned}		 
	\end{equation*}
	and the right hand side of the previous is bounded thanks to Lemma \ref{stimaloc}. 
\end{proof}
\begin{lemma}\label{lemmau}
	Under the assumptions of Lemma \ref{stimaloc} there exists $u\in L^q(0,T; W^{1,q}_{\rm loc}(\Omega))$ for any $q<p-\frac{N}{N+1}$ such that, up to a subsequence, $u_n$ converges to $u$ a.e. on $Q$,  weakly in $L^q(0,T; W^{1,q}_{\rm loc}(\Omega))$ and strongly in $L^1(0,T; L^1_{\rm loc}(\Omega))$. Moreover $\nabla u_n$ converges almost everywhere to $\nabla u$ in $Q$. In particular  $a(x,t,\nabla u_n)$  strongly converges to $a(x,t,\nabla u)$  in $L^1(0,T; L^1_{\rm loc}(\Omega))$.  
\end{lemma}
   \begin{proof}
   	From Lemma \ref{stimaloc} we know that $u_{n}$ is bounded in $L^q(0,T; W^{1,q}_{\rm loc}(\Omega))$ with $q<p-\frac{N}{N+1}$. Moreover from Corollary \ref{stimaL1} one has that the right hand side of \eqref{pbapprox} is bounded in $L^1(0,T; L^1_{\rm loc}(\Omega))$. Hence, let $\varphi \in C^1_c(\Omega)$ then one has that $(u_n\varphi)_t$ is bounded in $L^{s}(0,T;W^{-1,s}(\Omega)) + L^1(Q)$ with $s=\frac{q}{p-1}$, 
   which is sufficient to apply   \cite[Corollary $4$]{simon} in order to deduce that $u_n$ converges to a function $u$ in $L^1(0,T; L^1_{\rm loc}(\Omega))$. Furthermore, since the right hand side of \eqref{pbapprox} is bounded in $L^1(0,T; L^1_{\rm loc}(\Omega))$ and $u_n$ is bounded in $L^q(0,T; W^{1,q}_{\rm loc}(\Omega))$, one can reason as in the proof of   Theorem $4.1$ of \cite{bm} (see also \cite[Theorem 3.3]{bdgo}) deducing that $\nabla u_n$ converges to $\nabla u$ almost everywhere in $Q$. The last sentence of the statement  can be  checked by using  \eqref{cara2} and  Vitali's theorem  in order to show that  $|\nabla u_n|^{p-1}$ is strongly compact in  $L^1(0,T; L^1_{\rm loc}(\Omega))$. 
  \end{proof}
\subsection{Passing to the limit} \label{pass}
Here, using  the results obtained in Section \ref{approx_sec},  we pass to the limit in \eqref{pbapprox} in order to prove Theorem \ref{theo1}. 
\begin{proof}[Proof of Theorem \ref{theo1}]
We want to show that $u$, i.e. the almost everywhere limit of $u_n$ found in Lemma \ref{lemmau}, is  a solution to \eqref{pb}. We first want to obtain \eqref{weakdef3} by passing to the limit in 
\begin{equation*}\label{solapprox} 
\displaystyle -\int_{Q} u_n\varphi_t  - \int_{\Omega} u_{0_n}\varphi(x,0) + \int_{Q}a(x,t,\nabla u_n) \cdot \nabla \varphi  =\int_{Q} h_n(u_n)f_n\varphi  +\int_{Q} \mu_n \varphi,
\end{equation*}
where $\varphi \in C^1_c(\Omega\times [0,T))$. 
Since $u_n$ is strongly compact in $L^1(0,T; L^1_{\rm loc}(\Omega))$ and using the definition of $u_{0_n}$,  we easily  pass to the limit in the first two terms of the previous equality. We can also pass to the limit in  the the third term by using Lemma \ref{lemmau}. By the definition of  $\mu_n$ we also  pass to the limit in the last  term.

We are left to pass to the limit in the nonlinear lower order  term involving $h$. If $h(0)<\infty$ we use Lebesgue's dominated convergence theorem  and we easily pass $n$ to the limit.

 Hence,  assume $h(0)=\infty$. Let  us take $\varphi\ge 0$  (a  standard density argument will allow to deduce the result also for sign changing test functions). For $\delta >0$  we split the term  as 
\begin{equation}\label{rhs}
\int_{Q} h_n(u_n)f_n\varphi = \int_{Q\cap \{u_n\le \delta\}}h_n(u_n)f_n\varphi + \int_{Q\cap \{u_n> \delta\}}h_n(u_n)f_n\varphi,
\end{equation}
Without losing generality we may assume the parameter $\delta$ running outside   the set $ \{\eta: |\{u(x,t)=\eta \}|>0\}$ which is at most a countable.
The second term in \eqref{rhs} passes to the limit  again  by the Lebesgue theorem as 
$$h_n(u_n)f_n\varphi\chi_{\{u_n> \delta\}} \le \sup_{s\in [\delta,\infty)}[h(s)]\ f\varphi \in L^1(Q),$$
namely
\begin{equation*}\label{rhs2}
\lim_{n\to \infty}\int_{Q\cap \{u_n> \delta\}}h_n(u_n)f_n\varphi= \int_{Q\cap \{u> \delta\}}h(u)f\varphi.
\end{equation*}
 Let us observe that using the Fatou lemma and Corollary  \ref{stimaL1} imply $h(u)f\in L^1(0,T;L^1_{\rm loc}(\Omega))$.
This allows to apply once again the Lebesgue theorem as $\delta \to 0$ obtaining
\begin{equation*}\label{rhs21}
\lim_{\delta\to 0}\lim_{n\to \infty}\int_{Q\cap \{u_n> \delta\}}h_n(u_n)f_n\varphi= \int_{Q\cap\{u> 0\}}h(u)f\varphi.
\end{equation*}	
We also observe that $h(u)f\in L^1(0,T;L^1_{\rm loc}(\Omega))$ gives that the set  $\{u=0\}$ is contained in the set $\{f=0\}$ up to a set of zero Lebesgue measure. This means that 
$$\int_{Q\cap\{u> 0\}}h(u)f\varphi=\int_{Q}h(u)f\varphi,$$
and then the proof is done once we have shown that the first term in the right hand side of  \eqref{rhs} converges to zero as, resp.,   $n\to\infty$ and as $\delta \to 0$.
To this aim we take $V_\delta(u_n)\varphi$ ($V_\delta$ is defined in \eqref{Vdelta}) as test function in \eqref{pbapprox}. Dropping a negative term, one has 
\begin{equation*}\label{limn1}
\begin{aligned}
\int_{Q\cap \{u_n\le \delta\}}h_n(u_n)f_n\varphi &\le \int_0^T \langle(u_n)_t,V_\delta(u_n)\varphi\rangle + \int_{Q}|\nabla u_n|^{p-1}|\nabla \varphi| V_{\delta}(u_n)
\\
& \le -\int_\Omega \Phi_\delta (u_n)\varphi_t + \int_{Q}|\nabla u_n|^{p-1}|\nabla \varphi| V_{\delta}(u_n),
\end{aligned}
\end{equation*}
where $\Phi_\delta (s) = \int_0^s V_\delta(t) \ dt$. Furthermore, from the fact that $\int_\Omega \Phi_\delta (u_n) \le \delta$ and from \eqref{stimatk}, one is able to deduce that
\begin{equation*}
\begin{aligned}
 \limsup_{n\to\infty}\int_{Q\cap \{u_n\le \delta\}}h_n(u_n)f_n\varphi &\le C\delta^{\frac{1}{p'}},
\end{aligned}
\end{equation*}   
which implies that
\begin{equation*}
\begin{aligned}
\lim_{\delta\to 0}\limsup_{n\to \infty}\int_{Q\cap \{u_n\le \delta\}}h_n(u_n)f_n\varphi =0.
\end{aligned}
\end{equation*}  
Hence, we deduce that 
$$\lim_{n\to \infty}\int_{Q}h_n(u_n)f_n\varphi = \int_{Q}h(u)f\varphi,$$
for every $\varphi\in C^1_c(\Omega\times [0,T))$ and this proves that $u$ satisfies \eqref{weakdef3}. 

\medskip

In order to conclude the proof of Theorem \ref{theo1} we show that the boundary condition holds; namely $u$ satisfies \eqref{weakdef2}. By taking $T_k^\sigma(u_n)$ ($k>0$) as a test function in \eqref{pbapprox} one can  deduce that $T_k^{\frac{\sigma -1+p}{p}}(u_n)$ is bounded in $L^p(0,T; W^{1,p}_0(\Omega))$ with respect to $n$. Hence one  has that, for almost every $t\in (0,T)$ and for every $k>0$, $T_k^{\frac{\sigma -1+p}{p}}(u(x,t)) \in W^{1,p}_0(\Omega)$. If $\gamma\le1$ then $T_k(u(x,t)) \in W^{1,p}_0(\Omega)$ is known to imply \eqref{weakdef2}. Otherwise, if $\gamma>1$, one reasons as follows
\begin{equation*}
\begin{aligned}
\frac{1}{\epsilon} \int_{\Omega_\epsilon} T_k(u) \le \frac{1}{\epsilon} \left(\int_{\Omega_\epsilon} T_k^{\frac{\gamma -1+p}{p}}(u)\right)^{\frac{p}{\gamma-1+p}} |\Omega_\epsilon|^{\frac{\gamma-1}{\gamma-1+p}} = \left(\frac{|\Omega_\epsilon|}{\epsilon}\right)^{\frac{\gamma-1}{\gamma-1+p}}\left( \frac{1}{\epsilon}\int_{\Omega_\epsilon} T_k^{\frac{\gamma -1+p}{p}}(u)\right)^{\frac{p}{\gamma-1+p}},
\end{aligned}
\end{equation*}
and taking $\epsilon\to 0$ in the previous one has that \eqref{weakdef2} holds also  in this case. This concludes the proof.
\end{proof}
\subsection{Some comments on the regularity of the solutions}\label{comment}
In this section we show that the scheme of approximation \eqref{pbapprox} can take to a solution $u\in L^p(0,T; W^{1,p}_0(\Omega)) $ under suitable assumptions on the data, and in particular on the nonlinearity $h$ giving rise to some regularizing effects with respect to the purely quasilinear case (i.e. $h\equiv 1$). For a general nonhomogeneous datum $\mu$ then no solution belonging to  the natural energy space are expected, since only truncations of them are  (see \cite{op, op2} for similar considerations in the stationary case).  So that we restrict to the case $\mu\equiv 0$. 

 We assume the following control on $h$ at infinity
\begin{equation}
\exists \theta\ge 0, C,s_1>0: h(s)\le \frac{C}{s^\theta} \ \text{for all } s\ge s_1. 
\label{h2}
\end{equation}    
On the datum $f$ we assume that
\begin{equation}\label{f2}
\begin{aligned}
	&f \in L^{{\frac{p}{p-1+\theta}}}(0,T; L^{\left(\frac{p^*}{1-\theta}\right)'}(\Omega)) \ \ \ &\text{if} \ \theta<1, 
	\\
	&f \in L^{1}(Q) \ \ \ &\text{if} \ \theta\ge1. 	
\end{aligned}
\end{equation}

 \begin{lemma}\label{lemmaregolarita}
	Let $a$ satisfy \eqref{cara1}, \eqref{cara2}, \eqref{cara3}, let $h$ satisfy \eqref{h1} with $\gamma\le1$ and \eqref{h2}, let $f$ nonnegative  satisfy \eqref{f2}; finally let $u_0 \in L^2(Q)$ be  nonnegative. Then the solution to \eqref{pb} (with  $\mu\equiv 0$) found in Theorem \ref{theo1} belongs to $L^p(0,T; W^{1,p}_0(\Omega))$.
 \end{lemma}
\begin{proof}
	We need to prove that under our  assumptions  we can show better a priori estimates on the sequence $u_n$, i.e. a solution to \eqref{pbapprox}.
	Hence we take $u_n$ as a test function in \eqref{pbapprox} obtaining
	\begin{equation*}
	\begin{aligned}
		\alpha\int_{Q} |\nabla u_n|^p &\le \max_{s\in [0,s_1]} [h(s)s] \int_{Q\cap \{u_n\le s_1\}}f_n  + \int_{Q\cap \{u_n> s_1\}} f_n u_n^{1-\theta} +\frac{1}{2}\int_\Omega u_0^2
		\\
		&\le C + \int_{Q\cap \{u_n> s_1\}} f_n u_n^{1-\theta}
	\end{aligned}
	\end{equation*} 
	and hence if $\theta\ge 1$ the estimate is done since
	$$
	\int_{Q\cap \{u_n> s_1\}} f_n u_n^{1-\theta}\leq \int_{Q\cap \{u_n> s_1\}} f_n s_1^{1-\theta}\leq C\,. 
	$$
	
	  Otherwise we apply the H\"older, the Sobolev (with constant $\mathcal{S}_p$) and the Young inequalities to have 
	\begin{equation*}
\begin{aligned}
\alpha\int_{Q} |\nabla u_n|^p &\le \int_{0}^T ||f||_{L^{\left(\frac{p^*}{1-\theta}\right)'}(\Omega)}||u_n||^{1-\theta}_{L^{p^*}(\Omega)} \le \mathcal{S}_p^{1-\theta}\int_{0}^T ||f||_{L^{\left(\frac{p^*}{1-\theta}\right)'}(\Omega)} \left(\int_\Omega |\nabla u_n|^p\right)^{\frac{1-\theta}{p}}
\\
&\le C_\epsilon \mathcal{S}_p^{1-\theta}\int_{0}^T ||f||^{\frac{p}{p-1+\theta}}_{L^{\left(\frac{p^*}{1-\theta}\right)'}(\Omega)} + \epsilon\mathcal{S}_p^{1-\theta} \int_Q |\nabla u_n|^p, 
\end{aligned}
\end{equation*}	
and taking $\epsilon$ sufficiently small the estimate is closed also if $\theta<1$. Hence  $u_n$ is bounded in $L^p(0,T; W^{1,p}_0(\Omega))$ and so  $u$ belongs to $L^p(0,T; W^{1,p}_0(\Omega))$.
\end{proof}

\begin{remark}
We underline that requiring that $h$ goes to zero as in \eqref{h2} provides a strong regularizing effect on the Sobolev regularity of the solution to \eqref{pb}. Indeed when $h\equiv 1$ and $\mu\equiv 0$ then $u$ is, in general, expected to have finite energy  for a datum  $f\in L^{p'}(0,T; L^{(p^*)'}(\Omega))$  (i.e. $\theta=0$ in \eqref{h2}) (see \cite{lions}), while  for more general (e.g.  merely  integrable) data then $u$ needs not to (see \cite{bg}). In our case, even if $h$ is allowed to  mildly blow up at the origin ($\gamma \le 1$),  we are able to find more regular solution thanks to the regularizing effect given by the vanishing rate of  $h$ at infinity. This improved regularity is a key tool for uniqueness, since, as we will see in the next section, a unique solution can be proven to exist in a suitable subclass of this  {\it energy space}. Finally we underline that the exponent $\left(\frac{p^*}{1-\theta}\right)'$ is the same obtained in \cite{bo} when $p=2$ while studying the regularizing effect of the singular nonlinearity in the stationary case.
\end{remark}

\section{Uniqueness of finite energy solutions}\label{unicas}
In this section we prove uniqueness of finite energy solutions to

\begin{equation}
\begin{cases}
\displaystyle u_t -\operatorname{div}(a(x,t,\nabla u)) = h(u)f &  \text{in}\, Q, \\
u(x,t)=0 & \text{on}\ \Gamma,\\
u(x,0)=u_0(x) &  \text{in}\, \Omega\,; 
\label{pb0}
\tag{P$_0$}
\end{cases}
\end{equation}
provided $h$ is nonincreasing. We say that a distributional solution  $u$ of problem \eqref{pb0} is said to have  {\it  finite energy} provided $u\in L^{p}(0,T;W^{1,{p}}_0(\Omega))$ and  $u_t\in L^{p'}(0,T;W^{-1,{p'}}(\Omega)) + L^1(Q)$. We have the following 

\begin{theorem}\label{uni}
	Let $h$ be nonincreasing, then there is at most one  finite energy solution to \eqref{pb0}. 
\end{theorem}

The first step consists in extending the set of admissible test functions in \eqref{weakdef3}, namely we have

\begin{lemma} \label{este} Let $u$ be a finite energy solution to \eqref{pb0} then $u$ satisfies 
	\begin{equation} \label{def_estesa}
	\displaystyle \int_{0}^T \langle (u_t)_1, \varphi\rangle + \int_{Q} (u_t)_2\varphi + \int_{Q}a(x,t,\nabla u) \cdot \nabla \varphi  =\int_{Q} h(u)f\varphi,
	\end{equation}
	where $(u_t)=(u_t)_1+(u_t)_2$ with $(u_t)_1\in L^{p'}(0,T; W^{-1,{p'}}(\Omega))$,  $(u_t)_2\in L^1(Q)$ and for every $\varphi \in L^p(0,T; W^{1,p}_0(\Omega))\cap L^\infty(Q)$.
\end{lemma}
\begin{proof}
Let $\varphi\in L^p(0,T; W^{1,p}_0(\Omega))\cap L^\infty(Q)$ be  a nonnegative function and 
consider a sequence  $\varphi_n$  of smooth nonnegative functions converging to $\varphi$ in $L^p(0,T; W^{1,p}_0(\Omega))$, without loss of generality one may assume $||\varphi_n||_{ L^\infty(Q)}\leq  ||\varphi||_{ L^\infty(Q)}$. Now we take $\varphi_n$ 
in \eqref{weakdef3} (recall that $\mu\equiv 0$) and we integrate by parts (see for instance \cite{cw}) in order to obtain 
\begin{equation}\label{conn}
	\displaystyle \int_{0}^T \langle (u_t)_1, \varphi_n\rangle + \int_{Q} (u_t)_2\varphi_n + \int_{Q}a(x,t,\nabla u) \cdot \nabla \varphi_n  =\int_{Q} h(u)f\varphi_n.
\end{equation}
Due to the definition of $\varphi_n$ and on the regularity of $u$ one can pass to the limit in the left hand side of the previous equality; in particular  $ h(u)f\varphi_n$ is bounded in $L^1(Q)$ and by Fatou's lemma one obtains that also  $ h(u)f\varphi\in L^1(Q)$. Using dominated convergence theorem one can pass to the limit also on the right hand side of \eqref{conn} and we conclude. 
\end{proof}

\begin{proof}[Proof of Theorem \ref{uni}]
	
	Let $v,w$  be two finite energy solutions of  \eqref{pb0},  and take $T_k(v-w)\phi(t)$ in the difference of formulations \eqref{def_estesa} solved by $v, w$, where we  define  $\phi(t)= \frac{-t}{T} +1$ for $t\in (0,T]$. Using \eqref{cara3} and the assumption on $h$ we obtain that
	\begin{equation*}
		\begin{aligned}
		\displaystyle \int_{0}^T \langle (v_t)_1 - (w_t)_1, T_k(v-w)\phi(t)\rangle 
		+ \int_{Q} \left((v_t)_2 - (w_t)_2\right)T_k(v-w)\phi(t)   \le 0.	
		\end{aligned}
	\end{equation*} 
It follows from  of \cite[Lemma $7.1$]{dp} that  ($\tilde{T}_{k,1}(s)$ is defined in \eqref{Ttilde})
	\begin{equation*}
\begin{aligned}
\displaystyle \frac{1}{T}\int_Q \tilde{T}_{k,1}(v-w) = 0,	
\end{aligned}
\end{equation*} 	
that,  due to the arbitrariness of $k$  and recalling that $v-w \in C([0,T];L^1(\Omega))$,   implies that $v (\tau)=w(\tau)$  for any $\tau\in (0,T]$ and  for almost every $x$ in $\Omega$ and this concludes the proof.

\end{proof}

\begin{remark}
We want to highlight some cases in which the solution to \eqref{pb0} has finite energy. First of all observe that if $h(0)<\infty$  then Lemma \ref{lemmaregolarita} gives some instances  of a finite energy solutions (since, from the equation,  one has $u_t\in L^{p'}(0,T;W^{-1,{p'}}(\Omega)) + L^1(Q)$). Moreover if we restrict to data $a$ and $f$ not depending on $t$, and we consider  $u_0=w(x)$ where $w(x)$ is a  solution  to the associated elliptic problem 
$$
\begin{cases}
-\operatorname{div}(a(x,\nabla w)) = h(w)f &  \text{in}\, \Omega, \\
w=0 & \text{on}\ \partial\Omega\,,
\end{cases}
$$
then, if  $w \in W^{1,p}_0 (\Omega)$ (see \cite{lm,bo, op2}), one has, by Lemma \ref{este}, that $u(x,t)=w(x)$ turns out to be a finite energy solution of \eqref{pb0}. 

\end{remark}

Also in the general case of a nonlinearity $h$ blowing up at zero non-trivial non-stationary finite energy solutions do exist as the following example shows 
\begin{example}

We present a case where $h$ actually blows up at the origin and the solution to \eqref{pb0} has finite energy. Let $\gamma\le 1$ and, for the sake of simplicity let $p=2$, we consider the following problem
\begin{equation*}
\begin{cases}
\displaystyle u_t -\Delta u = f &  \text{in}\, Q, \\
u(x,t)=0 & \text{on}\ \Gamma,\\
u(x,0)=u_0(x) &  \text{in}\, \Omega\,,
\label{ex}
\end{cases}
\end{equation*}
where $0\leq f\in L^{2}(0,T; L^{(2^*)'}(\Omega))$ and $0\le u_0\in L^2 (\Omega)$. By the classical theory of   parabolic equations  one has that  a positive solution $u \in L^2(0,T; W^{1,2}_{0}(\Omega))$ does exist. We observe that $u$ also solves
\begin{equation*}
\begin{cases}
\displaystyle u_t -\Delta u = \frac{g}{u^\gamma} &  \text{in}\, Q, \\
u(x,t)=0 & \text{on}\ \Gamma,\\
u(x,0)=u_0(x) &  \text{in}\, \Omega\,,
\label{ex2}
\end{cases}
\end{equation*}	
where $g=fu^\gamma \in L^1(Q)$. Indeed, applying the H\"older and the Young inequalities, one has
	\begin{equation*}
\begin{aligned}
\int_{Q} fu^\gamma &\le \int_{0}^T ||f||_{L^{\left(\frac{2^{*}}{\gamma}\right)'}(\Omega)}||u||^{\gamma}_{L^{2^{*}}(\Omega)} \le \int_{0}^T ||f||^{\frac{2}{2-\gamma}}_{L^{\left(\frac{2^{*}}{\gamma}\right)'}(\Omega)} + \int_{0}^T ||u||^{2}_{L^{2^{*}}(\Omega)}.  
\end{aligned}
\end{equation*}	
A similar calculation gives the existence of a finite energy solution even if $\gamma>1$ provided we restrict to a datum  $ f\in L^{2}(0,T; L^{(\frac{2^*}{\gamma})'}(\Omega))$. Finally we also underline that in case $h(s)\approx s^{-\gamma}$ near $s=0$  with  $h(s)\ge c$ for all $s\ge 0$ then  a finite energy solution to \eqref{pb0} also exists, using that  $g=\frac{f}{h(u)}\leq \frac{f}{c}$.  
\end{example}

\begin{remark}
 For simplicity we presented our uniqueness result in the homogeneous case  $\mu\equiv 0$; as we have seen the key role in the proof is played by the possibility to extend the set of test functions (Lemma \ref{este}). For this reason, with the same proof, one can allow also nonhomogeneous measure data $\mu\in L^1 (Q)+L^{p'} (0,T; W^{-1,p'}(\Omega))$; as showed in \cite[Proposition 3.1]{ppp} (see also \cite[Theorem 2.6]{kr}) measure with such a structure are absolutely continuous with respect to the parabolic $p$-capacity. 
 \end{remark}

\section*{Acknowledgements}

\noindent The authors are  partially supported by  the Gruppo Nazionale per l'Analisi Matematica, la Probabilit\`a e le loro Applicazioni (GNAMPA) of the Istituto Nazionale di Alta Matematica (INdAM).

\end{document}